\documentclass[12pt, a4paper]{amsart}
\usepackage[dvips]{epsfig}
\usepackage{amsmath}
\usepackage{amssymb}
\usepackage{amsfonts}
\usepackage{dsfont}
\usepackage{bbm}
\usepackage{amsthm}
\usepackage{amsbsy}
\usepackage{amsgen}
\usepackage{amscd}
\usepackage{amsopn}
\usepackage{amstext}
\usepackage{amsxtra}
\usepackage[utf8]{inputenc}
\usepackage{enumerate}
\usepackage{hyperref}
\usepackage{lineno}
\usepackage{marginnote}
\usepackage{verbatim}
\usepackage{calrsfs}
\usepackage{times, graphicx}
\usepackage{eso-pic}
\usepackage{lastpage}
\DeclareMathAlphabet{\pazocal}{OMS}{zplm}{m}{n}
\usepackage[foot]{amsaddr}

\numberwithin{equation}{section}

\makeatletter
\newcommand*\rel@kern[1]{\kern#1\dimexpr\macc@kerna}
\newcommand*\widebar[1]{%
  \begingroup
  \def\mathaccent##1##2{%
    \rel@kern{0.8}%
    \overline{\rel@kern{-0.8}\macc@nucleus\rel@kern{0.2}}%
    \rel@kern{-0.2}%
  }%
  \macc@depth\@ne
  \let\math@bgroup\@empty \let\math@egroup\macc@set@skewchar
  \mathsurround\z@ \frozen@everymath{\mathgroup\macc@group\relax}%
  \macc@set@skewchar\relax
  \let\mathaccentV\macc@nested@a
  \macc@nested@a\relax111{#1}%
  \endgroup
}
\makeatother

\newtheorem{theorem}{Theorem}[section]
\newtheorem{proposition}[theorem]{Proposition}
\newtheorem{lemma}[theorem]{Lemma}

\theoremstyle{definition}
\newtheorem{definition}[theorem]{Definition}

\newtheorem{conjecture}[theorem]{Conjecture}

 \allowbreak

\begin{document}

\title[Automatic sequences]{A note on multiplicative automatic sequences}

\author{Oleksiy Klurman}
\email{lklurman@gmail.com}
\author{Pär Kurlberg}
\email{kurlberg@math.kth.se}\address{Department of Mathematics, KTH Royal Institute of Technology, Stockholm}

\date{\today}

\begin{abstract}
  We prove that any $q$-automatic completely multiplicative function
  $f:\mathbb{N}\to\mathbb{C}$ essentially coincides with a Dirichlet
  character. This answers a question of J. P. Allouche and
  L. Goldmakher and confirms a conjecture of J. Bell, N. Bruin and
  M. Coons for completely multiplicative functions.  Further, assuming
  two standard conjectures in number theory, the methods allows for
  removing the assumption of completeness.

\end{abstract}	
\maketitle
\section{Introduction}
Automatic sequences play important role in computer science and number
theory. For a detailed account of the theory and applications we refer
the reader to the classical monograph~\cite{AS}. One of the
applications of such sequences in number theory stems from a
celebrated theorem of Cobham~\cite{COB}, which asserts that in order
to show the transcendence of the power series $\sum_{n\ge 1}f(n)z^n$
it is enough to establish that the function$f:\mathbb{N}\to\mathbb{C}$
is {\it not} automatic. In this note, rather than working within the
general set up, we confine ourselves to functions with the range in
$\mathbb{C}.$ There are several equivalent definitions of automatic
(or more precisely, $q$-automatic) sequences. It will be convenient
for us to use the following one.
\begin{definition} The sequence $f:\mathbb{N}\to\mathbb{C}$ is called
  $q$-automatic if the $q$-kernel of it defined as a set of
  subsequences  
\[K_q(f)=\left\{ \{f(q^in+r\}_{n \ge 0} \vert\ i\ge 1,0\le r\le q^i-1\right\}\]
is finite. 
\end{definition}
We remark that any $q-$automatic sequence takes only finitely many
values, since it is a function on the states of finite automata.  A
function $f:\mathbb{N}\to\mathbb{C}$ is called completely
multiplicative if $f(mn)=f(m)f(n)$ for all $m,n\in\mathbb{N}.$ The
question of which multiplicative functions are $q$-automatic has been
the subject of study by several authors including~\cite{Yazd},
\cite{Puchta}, \cite{BBC}, \cite{Puchta1}, and \cite{AG}.
In particular, the following conjecture was made in \cite{BBC}.
\begin{conjecture}[Bell-Bruin-Coons]\label{Conj}
For any multiplicative $q$-automatic function
$f:\mathbb{N}\to\mathbb{C}$ there exists eventually periodic function
$g:\mathbb{N}\to\mathbb{C},$ such that $f(p)=g(p)$ for all primes $p.$  
\end{conjecture}
This conjecture is still open in general, although some progress has
been made when $f$ is assumed to be completely multiplicative. In
particular, Schlage-Puchta~\cite{Puchta} showed that a completely
multiplicative $q$-automatic sequence which does not vanish is almost
periodic. Hu~\cite{Hu} improved on that result by showing that the
same conclusion holds under a slightly weaker hypothesis. Our first
result confirms a strong form of Conjecture~\ref{Conj} when $f$ is
additionally assumed to be {\it completely} multiplicative function.
\begin{theorem}\label{main}
Let $q\ge 2$ and let $f:\mathbb{N}\to\mathbb{C}$ be completely
multiplicative $q$-automatic sequence. Then, there exists a Dirichlet
character of conductor $Q$ such that either $f(n)=\chi(n),$ for all
$(n,Q)=1$ or $f(p)=0$ for all sufficiently large $p.$\end{theorem} 
  We remark that similar result has been very recently
  obtained independently by Li~\cite{Li} using combinatorial methods
  relying on the techniques developed in the theory of automatic
  sequences. Our proof is shorter and builds upon two deep
  number theoretic results.
  Further, assuming the generalized Riemann hypothesis (which in
  particular implies a strong form of the Artin primitive root
  conjecture for primes in progressions) together with the set of
  base-$q$ Wieferich primes having density zero, our method can be
  adapted to show
  the full conjecture (i.e. the assumption on {\em complete}
  multiplicativity can be removed.)
\section{Proof of the main result}
We begin with a simple albeit important remark. Since $f$ is
$q$-automatic the image of $f:\mathbb{N}\to\mathbb{C}$ is finite and
therefore for any prime $p,$ $f(p)=0$ or $f(p)$ is a root of unity.
\begin{proposition}\label{key1}
Let $f:\mathbb{N}\to\mathbb{C}$ be a $q$-automatic completely
multiplicative function and let
$\mathcal{M}_0=\left\{p\vert\ f(p)=0\right\}.$
If $|\mathcal{M}_0|<\infty,$ then $f(p)=\chi(p)$ for all $p\notin \mathcal{M}_0.$
\end{proposition}
\begin{proof}
Since $f$ is $q$-automatic there exist positive integers $i_1\ne i_2,$
such that $f(q^{i_1}n+1)=f(q^{i_2}n+1)$ for all $n\ge 1.$ If
$n=m\prod_{p\in\mathcal{M}_0}p,$
then $$\frac{f(q^{i_1}m\prod_{p\in\mathcal{M}_0}p+1)}{f(q^{i_2}m\prod_{p\in\mathcal{M}_0}p+1)}=1\ne
0,$$ 
for all $m\ge 1.$ The conclusion now immediately follows from Theorem
$2$ of~\cite{EK}.\end{proof} 
Let $\overline{1,n}=[1,n]\cap\mathbb{Z}.$ Since $f$ is $q$-automatic,
there exists $k_0=k_0(f),$ such that for all $i\ge 1$ and $0\le r\le
q^i-1,$ the equalities $f(q^in+r)=0$ for $n \in \overline{1,k_0}$ imply
$f(q^in+r)=0$ for all $n\ge 1.$  
\begin{lemma}\label{key}
Suppose that $|\mathcal{M}_0|=\infty.$ For any
$q,k_0<p_1,p_2,\dots,p_{k_0}\in\mathcal{M}_0,$ there exists
$r=r(q,p_{1}, \ldots, p_{k_{0}})$ such
that $(r,q p_i)=1$ for all $i \in \overline{1,k_0}$ and
$f(n\prod_{i\le k_0}p_i+r)=0$ 
for all $n\ge 1.$
We may further assume that $r \equiv 3 \pmod {16}$, and $(r-1,\prod_{i\le k_0}p_i)=1.$
\end{lemma} 
\begin{proof}
  For an integer parameter $A\ge \log_q p_{k_0},$ which we shall
  choose later, by the Chinese remainder theorem there exists $r_A$
  such that $(r_A,q)=1$ and $r_A \equiv -sq^{2A}\pmod {p_s}$ for all
  $s \in \overline{1,k_0}.$
  Since $p_s\vert q^{2A}s+r_A$ we have
  $f(q^{2A}n+r_A)=0$ for all $n \in \overline{1,k_0}.$ The latter implies
  that $f(q^{2A}n+r_A)=0$ for all $n\ge 1.$ We claim that $f(r_A)=0.$
  Indeed, if this is not the case we choose a prime $p$, such that
  $f(p)=1$ and consider $m=p^{\phi(q^{2A})}r_A.$ Clearly
  $m \equiv r_A\pmod {q^{2A}}$ and consequently
  $0=f(m)=(f(p))^{\phi(q^{2A})}f(r_A)=1$, a contradiction. Note, that
 the same argument 
  works for $n\prod_{i\le k_0}p_i+r_A$ in place of $r_A$ and therefore
  we conclude that $f(n\prod_{i\le k_0}p_i+r_A)=0$ for all $n\ge 1.$
  Setting $r=r_A$ finishes the proof.
\end{proof}
Next, without loss of generality we may assume that there exist three
sufficiently large primes $t,t',t'' > \max(q,k_0)$ 
  such that $f(t)=f(t')=f(t'')=1.$
We will require the following consequence  of a result due to
Heath-Brown \cite{HB}.
\begin{lemma}
  \label{HBinAP}
  Given distinct primes $t,t',t'' > \max(q,k_0)$ and $r = r(q,p_{1},
  \ldots, p_{k_{0}})$ as in Lemma~\ref{key}, there exists
  infinitely many primes 
  $q_{i} \equiv r \pmod {16 \prod_{i\le k_0}p_i}$ such that at least one of
  $t,t',t''$ (say $t$) is a primitive root modulo $q_{i}$.  Moreover, by
  passing to a subsequence we may assume that for such primes
  $(q_{i}-1, q_{j}-1) = 2$ 
  for $i\neq j,$ and for each $\l \in \overline{1,k_0}$ we have
  $(\l/q_{i}) = 1$ for all $i \in \overline{1,k_0}$.
\end{lemma}
\begin{proof}
Let $v = 16 (\prod_{i\le k_0}p_i) \prod_{2<p\le k_0}p$ and chose $u$ such
that $u \equiv 3 \pmod {16}$ and 
$u \equiv r \pmod {\prod_{i\le k_0}p_i}$, with $r$ as in
Lemma~\ref{key}. Moreover, by  quadratic reciprocity we may further
select $u\pmod {\prod_{2<p\le k_0}p}$ such that $(u/p) = 1$ for all
primes $p\le k_0,$ 
and $(u-1,\prod_{2<p\le k_0}p)=1.$ In particular, we have
$(-3/p)=-1$ for any prime $p \equiv u \mod v$.
Applying Lemma~3 of \cite{HB}, with $u,v$ as above and $k=1$ (and
$K=2^{k}=2$) there exists $\alpha \in (1/4,1/2]$ and $\delta>0$ such 
that
$$
|\{ p \le x: p \equiv u\pmod v, (p-1)/K = P_{2}(\alpha,\delta) |
\gg
x/(\log x)^{2},
$$
with the implied constant possibly depending on $\alpha$, with
$P_{2}(\alpha,\delta)$ denoting the union of the set of primes,
together with the set of almost primes $n = t_{1}t_{2}$ with
$t_{1} < t_{2}$ both primes, and
$t_{1} \in [n^{\alpha}, n^{1/2-\delta}]$.  Heath-Brown's argument then
shows that at least one of $t,t',t''$ is a primitive root for
infinitely many primes $p \equiv u \pmod v$.  Whether the primes
$q_{i}$ produced have the properties that $(q_{i}-1)/2$ is prime, or
that $(q_{i}-1)/2 = t_{1} t_{2}$, we may pass to an infinite
subsequence of primes $q_{1} < q_{2} < \ldots$ (satisfying $q_{1}>q$)
so that $(q_{i}-1,q_{j}-1)=2$ for $i \neq j$ (for the latter case of
almost primes, note that both $t_{1}$ and $t_{2}$ are growing.)
\end{proof}

\begin{proposition}\label{key2}
Suppose that $|\mathcal{M}_0|=\infty.$ Then $f(p)=0$ for all
sufficiently large primes $p.$
\end{proposition}
\begin{proof} 
  Replacing $f$ by $|f|,$ which is also $q$-automatic, it is enough to
  prove the claim for the binary valued $f:\mathbb{N}\to \{0,1\}.$ By
  Lemma~\ref{HBinAP}, we may select prime $t$ with $f(t)=1$, which
  is a primitive root modulo infinitely primes $q_1<q_2\dots <q_{k_0}$
  (satisfying  $q_{1} > \max(k_0,q)$) such that $q_i \equiv r \pmod
  {16\prod_{j\le k_0}p_{k_0}}$ and consequently
  $f(q_{i})=0.$ From the proof of Lemma~\ref{key} it follows that
  there exists $r_A,$ such that $f(n\prod_{i\le
    k_0}q_i+r_A)=$ for all $n\ge 1.$ 
  Since $t$ is a primitive root modulo $q_j$ for  $j \in
  \overline{1,k_{0}}$, there exists 
  $\gamma_j$ such that $t^{\gamma_j} \equiv r_A\pmod {q_j}$ for
  $j \in \overline{1,k_0}.$ By the construction and Lemma~\ref{HBinAP} we have 
  $(r_{A}/q_{i}) = (-iq^{2A}/q_i)= -1$ and thus all $\gamma_{i}$ have the same
  parity. Consequently, by the Chinese remainder theorem we can choose $\gamma\in\mathbb{N},$ such that
  $\gamma \equiv \gamma_j\pmod {q_j-1}$ for all $j \in \overline{1,k_0}.$
  For $\gamma$ defined this way we have
  $t^{\gamma} \equiv r_A\pmod {\prod_{j\le k_0}q_j}.$ Hence, $f(t^{\gamma})$ must be zero. On the
  other hand $f(t^{\gamma})=f(t)^{\gamma}=1$, and this contradiction
  finishes the proof.\end{proof} 
  Combining Proposition~\ref{key1} and Proposition~\ref{key2} yields the conclusion of Theorem~\ref{main}
  \nocite{BCK}
\bibliographystyle{alpha} \bibliography{Automatabib}
  \end{document}